\documentclass[12pt]{article}
\usepackage{amsthm,amssymb,amsmath}
\usepackage{epsfig}

\title{Crossing numbers of periodic graphs}

\author{Zden\v{e}k Dvo\v{r}\'ak\thanks{Computer Science Institute of Charles University, Prague, Czech Republic.
E-mail: {\tt rakdver@iuuk.mff.cuni.cz}.
Supported by the Center of Excellence -- Inst. for Theor. Comp. Sci., Prague (project P202/12/G061 of Czech Science Foundation).}
\and
Bojan Mohar\thanks{Department of Mathematics, Simon Fraser University, Burnaby, B.C. V5A 1S6.
 E-mail: {\tt mohar@sfu.ca}.
 Supported in part by an NSERC Discovery Grant (Canada),
 by the Canada Research Chair program, and by the
 Research Grant P1--0297 of ARRS (Slovenia).}~\thanks{On leave from:
 IMFM \& FMF, Department of Mathematics, University of Ljubljana, Ljubljana,
 Slovenia.}}

\newtheorem{theorem}{Theorem}
\newtheorem{lemma}[theorem]{Lemma}
\newtheorem{corollary}[theorem]{Corollary}
\newtheorem{proposition}[theorem]{Proposition}
\newtheorem{conjecture}{Conjecture}

\DeclareMathOperator{\crn}{cr}
\newcommand\cyc{\odot}
\newcommand\GG{{\cal G}}

\begin{document}

\maketitle

%%%%%%%%%%%%%%%%%%%%%%%%%%%%%%%%%%%%%%%%%%
\begin{abstract}
A graph is \emph{periodic} if it can be obtained by joining identical pieces in a cyclic fashion.
It is shown that the limit crossing number of a periodic graph is computable. This answers a question of Richter~\cite[Problem~4.2]{BIRSreport}.
\end{abstract}
%%%%%%%%%%%%%%%%%%%%%%%%%%%%%%%%%%%%%%%%%%

\section{Introduction}

The asymptotic behavior of the crossing number of periodic graphs, i.e., the
graphs that can be obtained by joining identical pieces in a cyclic fashion,
plays fundamental role in constructions of crossing-critical graphs~\cite{PR2}
and in explaining certain phenomena. The first systematic treatment of this
area is due to Pinontoan and Richter~\cite{PR1}, who provided basic results and
motivated several questions. In this paper we provide a simplified, yet
equivalent setting for considering periodic graphs and answer
a question of Pinontoan and Richter.

We allow graphs to have loops and parallel edges.
A \emph{tile} is a triple $T=(G,A,B)$, where $G$ is a graph and $A=(a_1,a_2,\ldots, a_k)$ and $B=(b_1,b_2,\ldots, b_k)$
are sequences of vertices of $G$ of the same length $k$.  Note that each vertex of $G$ may appear several times
in $A$ and $B$.  The length $k$ of the sequences is called the \emph{width} of the tile, and the graph $G$ is
the \emph{underlying graph} of the tile.
If $T_1=(G_1,A_1,B_1)$ and $T_2=(G_2,A_2,B_2)$ are tiles of the same width, we denote by $T_1T_2$
the tile $(G,A_1,B_2)$, where $G$ is the graph obtained from the disjoint union of $G_1$ and $G_2$ by
adding, for $1\le i\le k$, an edge between the $i$-th vertices of $B_1$ and $A_2$.
By $\cyc(T_1)$, we denote the graph obtained from $G_1$ by adding edges between the $i$-th vertices of $A_1$ and $B_1$ for $i=1,\dots,k$.
If $T$ is a tile, we define $T^1=T$ and $T^n=T^{n-1}T$ for integers $n>1$.
We call the edges of $\cyc(T_1^n)$ that belong to the copies of $G_1$ \emph{internal}, and the edges between the copies \emph{external}.

For a tile $T$, we would like to determine the asymptotic behavior
of the crossing number of $\cyc(T^n)$.  Let us remark that this can significantly differ from
the crossing number of $T^n$ (compare e.g. the cylindrical $m\times n$ grid with the toroidal $m\times n$ grid).
Let $c_n(T)=\crn(\cyc(T^n))$.  Pinontoan and Richter proved in~\cite{PR1} that
the limit $$c(T)=\lim_{n\to\infty} \frac{c_n(T)}{n}$$ exists (they prove this only for tiles whose underlying graph is connected,
however the claim holds for disconnected tiles as well, see Lemma~\ref{lemma-limit}).
However, their proof gives no way to determine or to estimate the limit.\footnote{The abstract of \cite{PR1} erroneously states an estimate on the convergence rate of the limit, but the results of the paper do not give this, as confirmed in a private communication with the authors.}
Our main result is a bound on the convergence rate of the limit, giving an algorithm to approximate $c(T)$ within arbitrary precision.

\begin{theorem}
\label{thm:main}
For every tile $T$ and for every $\varepsilon>0$, there exists a computable constant
$N=O(1/\varepsilon^6)$ such that
$$
   \left|\frac{c_t(T)}{t} - c(T)\right| \le \varepsilon
$$
for every integer $t\ge N$.
\end{theorem}

Therefore, to estimate the limit within given precision $\varepsilon$, it suffices to determine the constant
$N$ of Theorem~\ref{thm:main} and compute the crossing number of $\cyc(T^N)$, which can be done using the
algorithm of Chimani et al.~\cite{crossalg}.  The resulting algorithm is exponential in a polynomial of $\frac{1}{\varepsilon}$.

It seems intuitively obvious that for large $n$, there is an optimal drawing of $\cyc(T^n)$ exhibiting some kind of
periodic behavior, and thus $c(T)$ is a rational number.  Nevertheless, we were not able to prove this, and we propose the following
conjecture.
\begin{conjecture}\label{conj-rat}
There exists a computable function $f$ that assigns to every tile $T$ a positive integer $f(T)$ such that
the number $c(T)$ is rational with denominator at most $f(T)$.
\end{conjecture}
If this were true, our result would give an algorithm to determine $c(T)$ exactly, by choosing $\varepsilon=\tfrac{1}{2f^2(T)}$
and rounding the result to the nearest rational number with denominator at most $f(T)$.
Note that $c(T)$ is not always integral, as shown by Pinontoan~\cite{Pinon06}.

Let us remark that we define the cyclic composition $\cyc(T^n)$ of copies of a tile
$T=(G,A,B)$ in a way slightly different from the definition of Pinontoan and Richter~\cite{PR1}.
They assume that the $2k$ vertices appearing in $A$ and $B$ are all distinct and that for $i,j=1,\dots,k$, $a_i$
and $a_j$ are adjacent if and only if $b_i$ and $b_j$ are adjacent. Then,
instead of adding the edges $a_ib_i$ joining consecutive copies of $T$, they
identify $a_i$ and $b_i$ for $i=1,\dots,k$.  Our choice of a different definition turns out to be very convenient in our proofs,
as we can treat the external edges added between the tiles specially.

It is easy to see that the same set of cyclic graphs can be constructed by both definitions.
Given a tile $T$ in our definition, we can add vertices $a_i'$ and $b_i'$ and edges $a_ia_i'$ and $b_ib_i'$
for $i=1,\dots k$, and after this change the construction of \cite{PR1} gives the same result as our
definition up to vertices of degree 2, which have no effect on the crossing number.
On the other hand, given the cyclic composition $H$ of $n$ copies of $T$ constructed
as in \cite{PR1}, we can choose an arbitrary edge-cut in $T$ separating $A$ from $B$
and split $H$ on the copies of this edge-cut, resulting in $n$ identical tiles which
can be composed to form $H$ using our definition of tile composition.

\section{Connectivity and linkedness}

Before starting with the main proofs, we show that it suffices to consider tiles that are connected and linked.
We say that a tile is \emph{connected} if its underlying graph is connected.
A tile $(G,A,B)$ of width $k$ is \emph{weakly linked}
if there exist pairwise edge-disjoint paths $P_1$, \ldots, $P_k$ in $G$ and a permutation $\pi$ of $[k]$ such
that for $1\le i\le k$, the path $P_i$ starts at the $i$-th element of $A$ and ends at the $\pi(i)$-th element of $B$.
A tile $(G,A,B)$ of width $k$ is \emph{linked}
if there exist pairwise edge-disjoint paths $P_1$, \ldots, $P_k$ in $G$ such
that for $1\le i\le k$, the endvertices of $P_i$ are the $i$-th elements of $A$ and $B$.

\begin{proposition}
\label{prop:1}
For every tile $T=(G,A,B)$ of width $k$ there exists a weakly linked tile $T_0$
of width $t\le k$ such that for every $n\ge 1$, we have
$\cyc(T_0^n)=\cyc(T^n)$.
\end{proposition}
\begin{proof}
Let $G'$ be the graph obtained from $G$ by adding new vertices
$a$ and $b$ and $k$ edges between $a$ and the elements of $A$ and $k$ edges between $b$ and the elements of $B$.
Let us first consider the case that there exists a minimal edge-cut $S=\{s_1,s_2,\ldots, s_t\}$ in $G'$ separating $a$ from $b$ with
$t<k$.  For $1\le i\le t$, let $u_i$ be the vertex incident with $s_i$ belonging to the component of $G'-S$ containing $a$,
if this vertex is different from $a$.  If $s_i$ is incident with $a$ and the edge $s_i$ joins $a$ with the $j$-th element of $A$,
then let $u_i$ be the $j$-th element of $B$.  Symetrically, let $v_i$ be the other vertex of $s_i$ if it is not equal to $b$, and
let $v_i$ be the corresponding
element of $A$ if $s_i$ is incident with $b$.  Let $T_0$ be the tile $(G'',(v_1,\ldots,v_t), (u_1,\ldots, u_t))$, where
$G''$ is the graph obtained from $\cyc(T)$ by removing the edges $u_1v_1, u_2v_2, \ldots, u_tv_t$. It is easy to see that
$\cyc(T_0^n)=\cyc(T^n)$ for every $n\ge 1$.

Since the width of $T_0$ is smaller than the width of $T$, we can perform this transformation only a bounded number of times.
Eventually, we obtain a tile $T_0 = (G'',(v_1,\ldots,v_t), (u_1,\ldots, u_t))$ of width $t\le k$ such that every edge-cut between the added vertices $a$ and $b$ has size at least $t$.
By Menger's theorem, there exist pairwise edge-disjoint paths $P_1$, \ldots, $P_t$ in $G''$ and a permutation $\pi\colon[t]\to [t]$
such that for $1\le i\le t$, the endvertices of $P_i$ are $v_i$ and $u_{\pi(i)}$. As mentioned above, we also have $\cyc(T_0^n)=\cyc(T^n)$ for every $n\ge 1$.
\end{proof}

Observe that if $T_1$ and $T_2$ are weakly linked tiles and $\pi_1$ and $\pi_2$ are the corresponding
permutations, then $T_1T_2$ is weakly linked for the permutation $\pi_2\circ\pi_1$.

\begin{proposition}\label{prop:3}
Let\/ $T$ be a weakly linked tile and let $\pi$ be the permutation from the definition of weakly linked tiles.
Let $m$ be the least common multiple of the lengths of the cycles of the permutation $\pi$. Then, the tile $T^m$ is linked.
\end{proposition}
\begin{proof}
As we observed, the tile $T^m$ is weakly linked for the permutation $\pi^m$, which is the identity permutation.
Consequently, $T^m$ is linked.
\end{proof}

The following proposition shows that the components of linked tiles are themselves tiles (this is not
necessarily the case in general; e.g., it is possible that a component contains different numbers of vertices
from $A$ and $B$).

\begin{proposition}\label{prop:2}
Let\/ $T=(G,A,B)$ be a linked tile, where $G$ is the disjoint union of graphs $G_1$ and $G_2$.
Let $A_1$ and $A_2$ be the sequences of vertices of $A$ in $G_1$ and $G_2$, respectively,
in the same order as in $A$.
Let $B_1$ and $B_2$ be the sequences of vertices of $B$ in $G_1$ and $G_2$, respectively,
in the same order as in $B$.
Then $T_1=(G_1,A_1,B_1)$ and $T_2=(G_2,A_2,B_2)$ are linked tiles
and for every $n\ge 1$, $\cyc(T^n)$ is the disjoint union of $\cyc(T_1^n)$ and $\cyc(T_2^n)$.
\end{proposition}
\begin{proof}
Let $k$ be the width of $T$ and consider some $i\in\{1,\ldots, k\}$.  Let $a_i$ and $b_i$ be the $i$-th vertices of $A$ and $B$, respectively,
and let $P_i$ be the path from the definition of linkedness joining $a_i$ with $b_i$.
Note that either $P_i\subseteq G_1$ or $P_i\subseteq G_2$, and in particular $a_i$ belongs to $G_1$ if and only if $b_i$ belongs to $G_1$.
We conclude that $T_1$ and $T_2$ are linked tiles.  Furthermore, the edges between the same vertices are added in the
constructions of $\cyc(T^n)$ and of $\cyc(T_1^n)$ and $\cyc(T_2^n)$, and thus $\cyc(T^n)=\cyc(T_1^n)\cup \cyc(T_2^n)$.
\end{proof}

Let $T=(G,A,B)$ be a tile of width $k$.  Let $H$ be a graph with vertices $v_1$, \ldots, $v_k$ and no edges, and let $Z$ be the tile $(H,(v_1,\ldots, v_k), (v_1,\ldots, v_k))$.
Let $Z'$ be a copy of the tile $Z$ with vertices $v'_1$, \ldots, $v'_k$.
A \emph{tile drawing of $T$} is a drawing of $ZTZ'$ in a closed disk such that the vertices $v_1,\ldots, v_k, v'_k, v'_{k-1}, \ldots, v'_1$ are drawn in the boundary of the disk in order.

We define $M(T)=\binom{|E(G)|+2k}{2}$.  Note that there exists a tile drawing of $T$ such that any two edges cross at most once; hence,
this drawing has at most $M(T)$ crossings.  By connecting $n$ such tile drawings into a cycle, we conclude that $c_n(T)\le M(T)n$ for every $n\ge 1$.
We often use variants of the following useful observation (an analogous result with $s=1$ was proved by Pinontoan and Richter~\cite{PR1}).

\begin{lemma}\label{lemma-nearc}
Let $T$ be a tile of width $k$ and let $s\ge 1$ an integer such that $T^s$ is connected.  For every $n\ge s+1$, there exists
a tile drawing of $T^n$ with at most $c_n(T)+(8k+1)M(T)s+\binom{2k}{2}$ crossings.
\end{lemma}
\begin{proof}
Let $\GG$ be a drawing of $\cyc(T^n)$ with $c_n(T)$ crossings.  For $1\le i\le n$, let
$a_i$ denote the number of crossings involving edges incident with the vertices of the $i$-th copy of $T$ in $\GG$,
and let $a'_i=\sum_{j=i}^{i+s-1} a_i$, where $a_{n+x}=a_x$ for $x\ge 1$.
Note that every crossing contributes $1$ to at most four of the numbers $a_i$, and thus
$\sum_{i=1}^n a_i\le 4c_n(T)\le 4M(T)n$ and $\sum_{i=1}^n a'_i\le 4M(T)ns$.  Hence, without loss of generality
we have $a'_1\le 4M(T)s$.

Let $S_1$ be the set of external edges of $\GG$ drawn between the first and the last copy of $T$, and
let $S_2$ be the set of external edges of $\GG$ drawn between the $s$-th and the $(s+1)$-th copy of $T$.
Let $v$ be any vertex of the first tile, and for each $e\in S_1\cup S_2$, let $P_e$ be a path
starting with $e$ and ending in $v$ contained in the first $s$ copies of $T$ (which exists since $T^s$ is connected).

Let $\GG_1$ be the drawing of $T^{n-s}$ obtained from $\GG$ by removing the first $s$ copies of $T$.
The drawing $\GG_1$ has at most $c_n(T)$ crossings.
Let $\GG_2$ be the drawing obtained from $\GG_1$ by adding back the vertex $v$ and for each $e\in S_1\cup S_2$, adding
an edge $e'$ between $v$ and the endvertex of $e$ contained on $\GG_1$, such that $e'$ is drawn along the path $P_e$
(perturbed slightly to avoid edges intersecting in infinite number of points, or three edges intersecting in one point).
Note that for each $e\in S_1\cup S_2$, every intersection of $e'$ with an edge $f$ of $\GG_1$ apears next to an intersection
of $f$ with $P_e$.  Consequently, $e'$ intersects $\GG_1$ in at most $a'_1$ points.

By splitting the vertex $v$ into $2k$ vertices of degree one and shifting the vertices slightly, we can transform
$\GG_2$ into a tile drawing $\GG_3$ of $T^{n-s}$ that extends $\GG_1$, without creating any new intersections with $\GG_1$.
Let $S'=\{e':e\in S_1\cup S_2\}$.  If two edges of $S'$ intersect more than once in the tile drawing $\GG_3$, we can eliminate the two crossings by swapping the parts of the edges between these crossings
and shifting the edges at the crossings slightly (this does not affect the crossings with other edges).
Furthermore, we can eliminate the crossings of edges of $S'$ with themselves by removing parts of the edges.
This way, we transform $\GG_3$ into a tile drawing $\GG_4$ such that no edge of $S'$ intersects itself and each two edges of $S'$ intersect
at most once.  Hence, $\GG_4$ has at most $c_n(T)+2ka'_1+\binom{2k}{2}\le c_n(T)+8kM(T)s+\binom{2k}{2}$ crossings.

By combining $\GG_4$ with $s$ tile drawings of $T$ (each with at most $M(T)$ crossings), we obtain the required tile drawing of $T^n$ with at most
$c_n(T)+(8k+1)M(T)s+\binom{2k}{2}$ crossings.
\end{proof}

For an integer $n\ge 1$, let $t_n(T)$ denote the minimum number of crossings in a tile drawing of $T^n$.
Pinontoan and Richter~\cite{PR1} observed that $t_n(T)$ is subadditive (i.e., for every $n_1,n_2\ge 1$, we have $t_{n_1+n_2}(T)\le t_{n_1}(T)+t_{n_2}(T)$,
and by Fekete's subadditive lemma, the limit $\lim_{n\to\infty} t_n(T)/n$ exists.
Note that a tile drawing of $T^n$ can be
turned into a drawing of $\cyc(T^n)$ with the same number of crossings by identifying the corresponding vertices in the boundary of the drawing
and suppresing the resulting vertices of degree two.  Hence, we have $c_n(T)\le t_n(T)$, and Lemma~\ref{lemma-nearc} gives a rough converse.
Hence, we obtain the following.

\begin{corollary}\label{cor-limac}
Let $T$ be a tile.  If there exists an integer $s\ge 1$ such that $T^s$ is connected, then the limit
$c(T)=\lim_{n\to\infty} c_n(T)/n$ exists and is equal to $\lim_{n\to\infty} t_n(T)/n$.
\end{corollary}

A simple consequence of Lemma~\ref{lemma-nearc} is the following relationship between $c_n(T)/n$ and $c_m(T)/m$ for $m\gg n$.

\begin{lemma}\label{lemma-upper}
Let $T=(G,A,B)$ be a connected tile of width $k$ and let $\varepsilon>0$ be a real number.
Let $n_2=2\left((8k+1)M(T)+\binom{2k}{2}\right) / \varepsilon$ and $a_0=2M(T)/\varepsilon$.
If $n\ge n_2$ and $m\ge a_0n$, then $c_m(T)/m\le c_n(T)/n+\varepsilon$.
\end{lemma}

\begin{proof}
By Lemma~\ref{lemma-nearc} applied with $s=1$, there exists a tile drawing $\GG_1$ of $T^n$ with at most $c_n(T)+(8k+1)M(T)+\binom{2k}{2}$ crossings.
Let $\GG_2$ be a tile drawing of $T$ with at most $M(T)$ crossings.  Suppose that $m=an+b$, where $a,b$ are nonnegative integers
and $b\le n-1$.  Let $\GG$ be the drawing of $\cyc(T^m)$ obtained by combining $a$ copies of $\GG_1$ and $b$ copies of $\GG_2$.
Then, $\GG$ has at most $a\bigl(c_n(T)+(8k+1)M(T)+\binom{2k}{2}\bigr)+bM(T)\le \tfrac{m}{n}c_n(T)+\tfrac{m}{n}\bigl((8k+1)M(T)+\binom{2k}{2}\bigr)+nM(T)$
crossings.  Consequently,
$$\frac{c_m(T)}{m}\le \frac{c_n(T)}{n}+\frac{1}{n}\left((8k+1)M(T)+\binom{2k}{2}\right)+\frac{n}{m}M(T)\le \frac{c_n(T)}{n}+\varepsilon.$$
\end{proof}

\begin{corollary}\label{cor-upper}
Let $T$ be a connected tile and let $\varepsilon>0$ be a real number.
Let $n_2$ be as in Lemma~\ref{lemma-upper}.  If $n\ge n_2$, then $c(T)\le c_n(T)/n+\varepsilon$.
\end{corollary}

Before we proceed with the proof of the main theorem, let us argue that the limit $c(T)$ exists
even if $T$ is not connected.

\begin{lemma}\label{lemma-limit}
For every tile $T=(G,A,B)$ of width $k$, the limit $c(T)=\lim_{n\to\infty} c_n(T)/n$ exists.
Furthermore, there exist integers $m\le k!$ and $r\le m|V(G)|$ and connected linked tiles $T_1$, \ldots, $T_r$
with at most $m|V(G)|$ vertices such that $c(T)=(c(T_1)+\ldots+c(T_r))/m$.
\end{lemma}
\begin{proof}
By Proposition~\ref{prop:1}, we can assume that $T$ is weakly linked.  Let $m$ be as in Proposition~\ref{prop:3}
and let $T_1$, \ldots, $T_r$ be the maximal connected subtiles of $T^m$.  By Proposition~\ref{prop:2},
these subtiles are linked.  Furthermore, the limits $c(T_1)$, \ldots, $c(T_r)$ exist by Corollary~\ref{cor-limac}.

For $1\le i\le r$, let $T'_i$ be the subtile of $T_i$ contained in the first copy of $T$ in $T^m$.
There exists a permutation $\pi$ of $[r]$ such that $T_i=T'_iT'_{\pi(i)}T'_{\pi^2(i)}\ldots T'_{\pi^{m-1}(i)}$.
Let $t$ be the number of cycles of $\pi$, and for $1\le i\le t$, let $\ell_i$ be the length of the $i$-th cycle of $\pi$, let $b_i$
an arbitrary element of the $i$-th cycle of $\pi$, and let $S_i=T'_{b_i}T'_{\pi(b_i)}T'_{\pi^2(b_i)}\ldots T'_{\pi^{\ell_i-1}(b_i)}$.
Then, for every $n\ge m$, $\cyc(T^n)$ is the disjoint union consisting of $\gcd(n,\ell_i)$ copies of $\cyc\left(S_i^{n/\gcd(n,\ell_i)}\right)$ for
$i=1,\ldots,t$.  Note that $S_i^{m/\ell_i}=T_{b_i}$ is connected.
Hence, by Corollary~\ref{cor-limac}, the limit $c(S_i)$ exists.

Consider any $\varepsilon>0$, and let $n_0\ge m$ be large enough that $|c_n(S_i)/n-c(S_i)|\le \varepsilon/t$ for every $n\ge n_0$ and $1\le i\le t$.
Let us remark that $\ell_i$ divides $m$, and thus $\ell_i\le m$.
Hence, for any $n\ge n_0m$, we have
\begin{eqnarray*}
\left|\frac{c_n(T)}{n}-\sum_{i=1}^t c(S_i)\right| &=& \left|\sum_{i=1}^t \frac{\gcd(n,\ell_i)}{n} \, c_{n/\gcd(n,\ell_i)}(S_i)-c(S_i)\right|\\
&\le&\sum_{i=1}^t\left|\frac{c_{n/\gcd(n,\ell_i)}(S_i)}{n/\gcd(n,\ell_i)} - c(S_i)\right|\\[1mm]
&\le&\varepsilon.
\end{eqnarray*}
We conclude that $c(T)=\lim_{n\to\infty} c_n(T)/n$ exists (and is equal to $\sum_{i=1}^t c(S_i)$).
Clearly, $c(T)=c(T^m)/m$, and by Proposition~\ref{prop:2}, we have $c(T^m)=\sum_{i=1}^r c(T_i)$.
Hence, $c(T)=(c(T_1)+\ldots+c(T_r))/m$ as required.
\end{proof}

\section{Drawings of periodic graphs}

Let $T=(G,A,B)$ be a tile and $n$ an integer.  Consider a drawing $\GG$ of $\cyc(T^n)$ and let $\beta>0$ be a real number.
The \emph{weight} of a crossing between two internal edges is $1+2\beta$, between an internal and an external edge
is $1+\beta$ and between two external edges is $1$.
Let $\crn_\beta(\GG)$ be the sum of the weights of the crossings in the drawing $\GG$.
The reason for introducing this weighted crossing number is the following lemma.

\begin{lemma}\label{lemma-elimbad}
Let $T=(G,A,B)$ be a linked tile of width $k$, let $n$ be an integer and let $\alpha\ge0$, $Q$ and $\beta>0$ be real numbers.
Let $n\ge 2$ be an integer and let $\GG$ be a drawing of $\cyc(T^n)$ such that $\crn_\beta(\GG)\le \alpha\, n+Q$.
If the internal edges of some copy of $T$ in the drawing $\GG$ participate in
at least $\tfrac{1}{\beta}\bigl(\binom{k}{2}+\alpha\bigr)$ crossings,
then there exists a drawing\/ $\GG'$ of\/ $\cyc(T^{n-1})$ with $\crn_\beta(\GG')\le \alpha (n-1)+Q$.
\end{lemma}

\begin{proof}
Let $c$ be the number of crossings in $\GG$ involving the internal edges of a copy $T'$ of $T$ in $\GG$,
and suppose that $c\ge \tfrac{1}{\beta}\bigl(\binom{k}{2}+\alpha\bigr)$.
Let $P_1$, \ldots, $P_k$ be the drawings of the paths from the definition of the linkedness of $T'$ according to $\GG$.
Let $\GG'_0$ be the drawing of $\cyc(T^{n-1})$ obtained from $\GG$ by removing $T'$ and by connecting the corresponding external edges incident with $T'$ along the paths $P_1$, \ldots, $P_k$.
In case that more than two of the paths intersect in a vertex of the underlying graph of $T'$, we shift the new edges slightly so that at most two edges
intersect in each point.  Let $S$ be the set of the resulting edges.
Note that each of the crossings with the internal edges of $T'$ in $\GG$ either disappears or becomes a crossing with an external edge of $S$ in $\GG'_0$,
and thus its weight is decreased at least by $\beta$.
On the other hand, by drawing the edges along the paths $P_1$, \ldots, $P_k$, we could introduce crossings between the edges of $S$; let $s$ be
the number of these new crossings.  This shows that $\crn_\beta(\GG'_0)\le \crn_\beta(\GG)-\beta c+s$.

If two edges of $S$ intersect more than once, we can eliminate the two crossings by swapping the parts of the edges between these crossings
and shifting the edges at the crossings slightly (this does not affect the crossings with other edges).
Furthermore, we can eliminate the crossings of edges of $S$ with themselves by removing parts of the edges.
In this way, we transform the drawing $\GG'_0$ to a drawing $\GG'$, where the number of crossings among the edges of $S$ is at most
$\binom{k}{2}$.  Thus, we have $\crn_\beta(\GG')\le \crn_\beta(\GG)-\beta c+\binom{k}{2}\le \crn_\beta(\GG)-\alpha\le \alpha (n-1)+Q$, as required.
\end{proof}

For a tile $T$ and two edges $e_1$ and $e_2$ of $\cyc(T^n)$, the \emph{cyclic tile distance} between $e_1$ and $e_2$
is the minimum number of distinct tiles that a path
in $\cyc(T^n)$ between the two edges must intersect.

\begin{lemma}
\label{lemma-nodist}
Let $T=(G,A,B)$ be a connected linked tile of width $k$ and let $\beta,\varepsilon>0$ and $\alpha\ge 0$ be real numbers.
Let $c=\tfrac{1}{\beta}\bigl(\binom{k}{2}+\alpha\bigr)$ and $Q_0=2k(2|E(G)|+2c+4k)(1+\beta)+4k^2+2\alpha$.
Let $n_1\ge n_0\ge 1$ be integers and let $Q\ge Q_0$ be a real number, and suppose that
$n$ is the smallest integer such that $n\ge n_1$ and there exists a drawing $\GG$ of $\cyc(T^n)$
with $\crn_\beta(\GG)\le \alpha n + Q$.
Furthermore, assume that
\begin{itemize}
\item $\GG$ is chosen among the drawings of $\cyc(T^n)$ so that $\crn_\beta(\GG)$ is the smallest possible, and
\item for $n_0\le m< n_1$, there is no drawing $\GG'$ of $\cyc(T^m)$ with $\crn_\beta(\GG')\le \alpha\, m + Q_0$.
\end{itemize}
If $n\ge 2n_1+2$, then the cyclic tile distance between any two crossing edges of $\GG$ is at most $n_0+1$.
\end{lemma}

\begin{proof}
Since $n>n_1$, the minimality of $n$ and Lemma~\ref{lemma-elimbad} imply that every copy of $G$ in $\GG$ is intersected at most $c$ times.

Suppose that $\GG$ contains two crossing edges $e_1$ and $e_2$ with cyclic tile distance at least $n_0+2$.
Let $\GG_1$ and $\GG_2$ be the subdrawings of $\GG$ consisting of the tiles on the two paths between $e_1$ and $e_2$ in the cycle forming $\cyc(T^n)$;
in case that $e_1$ or $e_2$ is an internal edge of a tile, this tile is included neither in $\GG_1$ nor in $\GG_2$.
For $i\in\{1,2\}$, $\GG_i$ is a drawing of $T^{k_i}$, where $k_1$ and $k_2$ are integers such that $k_1+k_2=n$ if both $e_1$ and $e_2$ are external,
$k_1+k_2=n-1$ if one of $e_1$ and $e_2$ is internal and $k_1+k_2=n-2$ if both $e_1$ and $e_2$ are internal.  Furthermore, we have $k_1,k_2\ge n_0$.
Since $n\ge 2n_1+2$, we can also assume that $k_2\ge n_1$.

Suppose that $e_1$ is an external edge, and let $S$ be the set of edges between the two copies of $G$ that are joined by $e_1$.
Let $x$ be the sum of weights of crossings of $e_1$ in $\GG$.
Consider an edge $e'\in S$ and let $y$ be the sum of weights of crossings of $e'$ in $\GG$.
Since $T$ is connected, we can redraw $e'$ to follow a path in the tiles with
which it is incident to come close to the ends of $e_1$ and then follow the
drawing of $e_1$ instead of its current drawing. Since every copy of $G$ is
intersected at most $c$ times, the crossings on the redrawn edge $e'$ would
have weight at most $(2|E(G)|+2c+4k)(1+\beta)+x$.
By the minimality of $\crn_\beta(\GG)$,
we have $(2|E(G)|+2c+4k)(1+\beta)+x\ge y$, and by symmetry, $|x-y|\le (2|E(G)|+2c+4k)(1+\beta)$.
Therefore, we can redraw all edges of $S$ along $e_1$ and increase $\crn_\beta(\GG)$ by at most $k(2|E(G)|+2c+4k)(1+\beta)$.
If $e_2$ is external, we perform a similar transformation for $e_2$ as well; this may incur an additional penalty of at most $4k^2$ for the intersections
between the rerouted edges. By allowing this penalty, we can achieve any specified order of the edges in $S$ close to the crossing of $e_1$ with $e_2$.
In case that $e_1$ or $e_2$ are internal, we can perform the same transformation after first eliminating the copies
of $G$ containing $e_1$ or $e_2$ similarly to the proof of Lemma~\ref{lemma-elimbad}.

Finally, we can match the rerouted edges in the vicinity of the crossing of $e_1$ and $e_2$
and obtain drawings $\GG_1'$ and $\GG'_2$ of $\cyc(T^{k_1})$ and $\cyc(T^{k_2})$, respectively, such that
$\crn_\beta(\GG_1')+\crn_\beta(\GG_2')\le \crn_\beta(\GG)+2k(2|E(G)|+2c+4k)(1+\beta)+4k^2=\crn_\beta(\GG)+Q_0-2\alpha$.
Since $n>k_2\ge n_1$, the minimality of $n$ implies that $\crn_\beta(\GG_2')>\alpha k_2+Q$.  Furthermore, note that
$\crn_\beta(\GG_1')>\alpha k_1+Q_0$ (this follows from the assumptions of the lemma if $k_1<n_1$, and
by the minimality of $n$ otherwise).  Therefore, $\crn_\beta(\GG_1')+\crn_\beta(\GG_2')>\alpha(k_1+k_2)+Q+Q_0\ge \alpha n + Q+Q_0-2\alpha$.
Therefore, we have $\crn_\beta(\GG)>\alpha n+Q$, which is a contradiction.
\end{proof}

The main approximation result follows straightforwardly from the next lemma.

\begin{lemma}\label{lemma-lbound}
Let $T=(G,A,B)$ be a connected linked tile of width $k$ and let $\varepsilon\le 1$ be a positive real number.
Let $\alpha=c(T)+\tfrac{1}{2}\varepsilon$, $\beta = \varepsilon/(8\alpha)$, $c=\tfrac{1}{\beta}\bigl(\binom{k}{2}+\alpha\bigr)$, $Q_0=2k(2|E(G)|+2c+4k)(1+\beta)+4k^2+2\alpha$,
$n_0 = \lceil 2Q_0/\varepsilon\rceil$, $Q=8c(n_0+1)(1+\beta)+4k^2(n_0+2)^2+2\binom{k}{2}$ and $n_1=\lceil 2Q/\varepsilon\rceil$.  Then, there exists $n$ such that $n_0\le n\le 2n_1+1$ and
$c_n(T)/n\le c(T)+\varepsilon$.
\end{lemma}

\begin{proof}
Since $\lim_{t\to\infty} \frac{c_t(T)}{t}=c(T)$, for every sufficiently large $t$, there exists a drawing $\GG$ of
$\cyc(T^t)$ with $\crn(\GG)\le (c(T)+\varepsilon/(4(1+2\beta)))t$.  Note that $\crn_\beta(\GG)\le (1+2\beta)\crn(\GG)$,
and thus $\crn_\beta(\GG)\le (c(T)+2\beta c(T)+\tfrac{1}{4}\varepsilon)t \le \alpha t$ for such an integer $t$.
Hence, there exists the smallest integer $n$ such that $n\ge n_1$ and for some drawing $\GG$ of $\cyc(T^n)$, we have
$\crn_\beta(\GG)\le \alpha n + Q$.
Let $\GG$ be such a drawing with the smallest possible $\crn_\beta(\GG)$.

If there exists $m$ such that $n_0\le m<n_1$ and there is a drawing $\GG'$ of $\cyc(T^m)$ with $\crn_\beta(\GG')\le \alpha m + Q_0$,
then the claim of the lemma holds, since
$\crn(\GG')\le \crn_\beta(\GG')\le \alpha m + Q_0=(c(T)+\tfrac{1}{2}\varepsilon+Q_0/m) m\le (c(T)+\tfrac{1}{2}\varepsilon+Q_0/n_0) m \le (c(T)+\varepsilon)m$.
Therefore, assume that there is no such $m$.

Similarly, if $n\le 2n_1+1$, then the claim holds, since
$\crn(\GG)\le\crn_\beta(\GG)\le \alpha n + Q=(c(T)+\tfrac{1}{2}\varepsilon+Q/n) n\le (c(T)+\tfrac{1}{2}\varepsilon+Q/n_1)n \le (c(T)+\varepsilon)n$.
Thus, we may assume that $n\ge 2n_1+2$ and we can apply Lemma \ref{lemma-nodist} to conclude that any two crossing edges in $\GG$ have cyclic tile distance at most $n_0+1$.

Let $k_1$ and $k_2$ be integers such that $k_1+k_2=n$ and $k_1, k_2\ge n_1$.
The drawing $\GG$ can be decomposed to drawings $\GG_1$ and $\GG_2$ of $T^{k_1}$ and $T^{k_2}$.
Consider the drawing $\GG_1$, and let $A_0$ and $B_0$ be the vertices of
the first and the last tile of $\GG_1$, respectively, incident with the external edges of $\GG\setminus \GG_1$.  Let $S$ be the set of edges of
$\GG\setminus \GG_1$ incident with $A_0\cup B_0$.  Let $Z$ be the set of edges of $\GG_1$
that are at cyclic tile distance at least $n_0+2$ from the edges of $S$ and let $R$ be the edges of $\GG_1$
at cyclic distance at most $n_0+1$ from $S$.  Since $T$ is linked, there exist pairwise edge-disjoint paths $P_1$, \ldots, $P_k$ in
$\GG_2\cup S$ joining the corresponding vertices of $A_0$ and $B_0$.  For $i=1,\ldots, k$, let us add an edge to $\GG_1$ between
the $i$-th vertices of $A_0$ and $B_0$ drawn along $P_i$;
let $\GG'_1$ be the resulting drawing and let $S'$ be the set of newly added edges.

By Lemma~\ref{lemma-nodist}, the edges of $S'$ do not intersect $Z$.
Let us estimate the weight of crossings between $S'$ and $R$. Note that
each such crossing corresponds to a crossing in $\GG$ between an edge of $R$ and of $\GG_2$.
The internal edges of $R$ belong to at most $2(n_0+1)$ distinct tiles of $\GG$, and thus
by Lemma~\ref{lemma-elimbad}, they are intersected at most $2c(n_0+1)$ times.
Consider the external edges of $R$.  By Lemma~\ref{lemma-nodist}, they
can only intersect the edges of $\GG_2$ at cyclic tile distance at most $n_0+1$ from $S$.
By Lemma~\ref{lemma-elimbad}, it follows that there are at most
$2c(n_0+1)$ intersections between external edges of $R$ and internal edges of $\GG_2$.
Finally, note that each two external edges of $\GG$ intersect at most once, as otherwise
we could decrease $\crn_\beta(\GG)$.  Therefore, the number of crossings between
external edges of $R$ and $\GG_2$ is bounded by $2k^2(n_0+2)^2$.

As usual, we can assume that each two edges of $S'$ intersect at most once in $\GG'_1$
by redrawing them if necessary.  We conclude that
$\crn_\beta(\GG'_1)\le\crn_\beta(\GG_1)+4c(n_0+1)(1+\beta)+2k^2(n_0+2)^2+\binom{k}{2}=\crn_\beta(\GG_1)+ \tfrac{1}{2}\,Q$.
Similarly, we can obtain a drawing $\GG'_2$ of $\cyc(T^{k_2})$ with $\crn_\beta(\GG'_2)\le \crn_\beta(\GG_2)+ \tfrac{1}{2}\,Q$.
We have $\crn_\beta(\GG_1)+\crn_\beta(\GG_2)\le \crn_\beta(\GG)$, and thus
$\crn_\beta(\GG'_1)+\crn_\beta(\GG'_2)\le \crn_\beta(\GG)+Q$.

However, since $n_1\le k_1,k_2<n$, the minimality of $n$ implies that
$\crn_\beta(\GG'_1)>\alpha k_1 + Q$ and
$\crn_\beta(\GG'_2)> \alpha k_2 + Q$.
Therefore, $\crn_\beta(\GG)>\alpha n+Q$, which is a contradiction.
\end{proof}

We are ready to give the proof of the main theorem.

\begin{proof}[Proof of Theorem \ref{thm:main}]
By Lemma~\ref{lemma-limit}, we can assume that the tile $T=(G,A,B)$ is connected and linked
(otherwise, we consider each connected subtile of $T^m$ separately).
Let $k$ be the width of $T$.  Let $\varepsilon_1$ be a constant to be chosen later and let us consider the following quantities:
\begin{itemize}
\item $\alpha_d=\tfrac{1}{2}\varepsilon_1$ and $\alpha_u=M(T)+\tfrac{1}{2}\varepsilon_1$
\item $\beta_d = \varepsilon_1/(8\alpha_u)$ and $\beta_u = \varepsilon_1/(8\alpha_d)$
\item $c_d=\tfrac{1}{\beta_u}\bigl(\binom{k}{2}+\alpha_d\bigr)$ and $c_u=\tfrac{1}{\beta_d}\bigl(\binom{k}{2}+\alpha_u\bigr)$
\item $Q_{0,d}=2k(2|E(G)|+2c_d+4k)(1+\beta_d)+4k^2+2\alpha_d$ and\\ $Q_{0,u}=2k(2|E(G)|+2c_u+4k)(1+\beta_u)+4k^2+2\alpha_u$
\item $n_{0,d} = \lceil 2Q_{0,d}/\varepsilon_1\rceil$ and $n_{0,u} = \lceil 2Q_{0,u}/\varepsilon_1\rceil$
\item $Q_u=8c_u(n_{0,u}+1)(1+\beta_u)+4k^2(n_{0,u}+2)^2+2\binom{k}{2}$ and
\item $n_{1,u}=\lceil 2Q_u/\varepsilon_1\rceil$.
\end{itemize}
These numbers are chosen in such a way that they are related to the values used in Lemma \ref{lemma-lbound}. More precisely, $\alpha_d,\beta_d,c_d,\dots$ give lower bounds and
$\alpha_u,\beta_u,c_u,\dots$ give upper bounds on
the corresponding quantities in Lemma \ref{lemma-lbound}, used with $\varepsilon_1$ in the role of $\varepsilon$.
In particular, if $n_0$ and $n_1$ are the constants of Lemma~\ref{lemma-lbound} for $T$ and $\varepsilon_1$, then
$n_0\ge n_{0,d}$ and $n_1\le n_{1,u}$.  Furthermore, $n_{0,d}=\Theta(1/\varepsilon_1)$ and $n_{1,u}=\Theta(1/\varepsilon_1^5)$ are computable, given $T$. (Here and below, all constants involved in the $\Theta$-notation depend only on $T$.)

Let $n_2=\Theta(1/\varepsilon)$ and $a_0=\Theta(1/\varepsilon)$ be the constants of Lemma~\ref{lemma-upper} applied for $\varepsilon/2$.
Let us now choose $\varepsilon_1 \le \varepsilon/2$ in such a way that $n_{0,d}\ge n_2$ and $\varepsilon_1 = \Theta(\varepsilon)$.
Let $N=a_0n_{1,u}$ and note that $N=\Theta(1/\varepsilon^6)$.

By Lemma~\ref{lemma-lbound}, there exists $n$ such that $n_{0,d}\le n\le n_{1,u}$ and $c_n(T)/n\le c(T)+\varepsilon_1\le c(T)+\varepsilon/2$.
By Lemma~\ref{lemma-upper}, we have $c_t(T)/t\le c_n(T)/n+\varepsilon/2\le c(T)+\varepsilon$ for every $t\ge N$.
Conversely, Corollary~\ref{cor-upper} implies that $c_t(T)/t\ge c(T)-\varepsilon$.
\end{proof}

\subsection*{Acknowledgement}
The authors are grateful to BIRS for providing research environment during a workshop on crossing numbers \cite{BIRSreport} where the results of this paper have emerged. The authors also acknowledge helpful discussions with several workshop participants within a work group where they presented their initial ideas.

\bibliographystyle{plain}
\bibliography{tiles}

\end{document}